\documentclass[12pt]{amsart}
\usepackage[latin1]{inputenc}
\usepackage{mathptmx}
\usepackage{amscd}
\usepackage{color}

\newtheorem{theorem}{Theorem}[section]
\newtheorem{lemma}[theorem]{Lemma}
\newtheorem{corollary}[theorem]{Corollary}
\newtheorem{proposition}[theorem]{Proposition}

\theoremstyle{definition}

\newtheorem{definition}[theorem]{Definition}

\let\frak=\mathfrak
\let\Bbb=\mathbb
\let\phi=\varphi

\def\Hom{\operatorname{Hom}}
\def\Ext{\operatorname{Ext}}

\def\Ass{\operatorname{Ass}}

\def\Spec{\operatorname{Spec}}
\def\depth{\operatorname{depth}}

\def\Ker{\operatorname{Ker}}

\def\Im{\operatorname{Im}}

\let\oldbigwedge\bigwedge
\def\BIGwedge{{\textstyle\oldbigwedge}}
\def\medwedge{{\scriptstyle\oldbigwedge}}
\def\bigwedge{\mathchoice{\BIGwedge}{\BIGwedge}{\medwedge}{}}

\DeclareMathOperator{\Max}{Max} \DeclareMathOperator{\Supp}{Supp}

\let\epsilon=\varepsilon
\let\tilde=\widetilde

\begin{document}

\title[Regular sequences and  local cohomology modules...]{Regular sequences and  local cohomology modules with
respect to a pair of ideals}
\author{  Sh. Payrovi }
\author{ M.  Lotfi Parsa }
\address{I. K. International University,
 Postal Code: 34149-1-6818
Qazvin - IRAN } \email{shpayrovi@ikiu.ac.ir} \email{
lotfi.parsa@yahoo.com}

\begin{abstract}
Let $R$ be a  Noetherian ring, $I$ and $J$  two ideals of $R$ and  $t$  an integer.
Let $S$ be the class of  Artinian $R$-modules, or  the class of all $R$-modules $N$ with $\dim_RN\leq k$,
where $k$ is an integer.
 It is proved that
$\inf\{i: H^{i}_{I,J}(M)\notin S\}=\inf\{S-\depth_\frak{a}(M):
\frak{a}\in \tilde{\rm W}(I,J)\}$,
where $M$  is a finitely generated $R$-module, or is a  $ZD$-module
 such that $M/\frak{a}M\notin S$
for all $\frak{a}\in \tilde{\rm W}(I,J)$.
Let $\Supp_R H^{i}_{I,J}(M)$ be a finite subset of $\Max(R)$ for
all $i<t$. It is  shown that there are maximal ideals $\frak m_1, \frak
m_2,\ldots,\frak m_k$ of $R$ such that
$H^{i}_{I,J}(M)\cong H^{i}_{\frak m_1}(M)\oplus H^{i}_{\frak
m_2}(M)\oplus\cdots\oplus H^{i}_{\frak m_k}(M)$ for all $i<t$.
\end{abstract}

{\footnote{$\hspace*{-5mm}$ 2010 Mathematics Subject
Classification. 13C15, 13D07, 13D45. \\Key words and Phrases.
Local cohomology, Regular sequence, Serre subcategory, $ZD$-module.} \vspace{-0.5cm}
\maketitle
\section{Introduction }

Throughout this paper, $R$ is a commutative Noetherian ring with
non-zero identity, $I$ and $J$ are two ideals of $R$, $M$ is an
$R$-module and  $s$ and $t$ are two   integers. For notations and
terminologies not given in this paper, the reader is referred to
\cite{bs}, \cite{bh} and \cite{ty} if necessary.

The theory of local cohomology, which was introduced by Grothendieck \cite{g}, is a useful
tool for attacking problems in commutative algebra and algebraic geometry.
Bijan-Zadeh  \cite{b} introduced  the local cohomology
modules with respect to a system of ideals, which is a generalization of ordinary local cohomology modules.
As a special case of these extend modules,
Takahashi, Yoshino and Yoshizawa \cite{ty} defined the local cohomology
modules with respect to a pair of ideals. To be
more precise, let ${\rm W}(I, J)=\{\frak p\in \Spec(R):
I^t\subseteq J+\frak p ~~~~\text{for some positive
integer}~~~~t\}$. The set of elements $x$ of $M$ such that
$\Supp_RRx\subseteq {\rm W}(I, J)$, is said to be $(I, J)$-torsion
submodule of $M$ and is denoted by $\Gamma_{I,J}(M)$.
 $\Gamma_{I,J}(-)$ is a covariant,
$R$-linear functor from the category of $R$-modules to itself. For
an integer $i$, the local cohomology functor $H^{i}_{I,J}(-)$ with
respect to $(I, J)$, is defined to be the $i$-th right derived
functor of $\Gamma_{I,J}(-)$. Also $H^{i}_{I,J}(M)$ is called the
$i$-th local cohomology module of $M$ with respect to $(I, J)$. If
$J=0$, then $H^{i}_{I,J}(-)$ coincides with the ordinary local
cohomology functor $H^{i}_{I}(-)$.  Let
$\tilde{{\rm W}}(I, J)=\{\frak a\unlhd R: I^t\subseteq J+\frak
a~~~~\text{for some positive integer} ~~~~t\}$.
 It is easy to see that
$$\Gamma_{I,J}(M)=\{x\in M: \exists\ \frak a\in {\rm \tilde W}(I,J),\ \frak ax=0\}
=\bigcup_{\frak a\in {\rm \tilde W}(I,J)}(0:_{M}\frak a).$$

In section 2, we study extension functors of local cohomology modules with respect to a pair
of ideals. Let $S$ be a Melkersson subcategory with respect to  $I$,
and $M$  a finitely generated $R$-module. The current authors, in
\cite[Theorem 2.11]{pl2}, showed that if $H^{i}_{I,J}(M)\in S$ for
all $i < t$, then $H^{i}_{I}(M)\in S$ for all $i < t$. In
\ref{16}, we improve this result and we show that if
$\Ext_R^{j}(N, H^{i}_{I,J}(M))\in S$ for all $i<t$ and all $j<
t-i$, then $H^{i}_{\frak{a}}(M)\in S$ for all $i<t$, where $M$ is
an arbitrary $R$-module and $N$ is  a finitely generated
$R$-module with $\Supp_R N={\rm V}(\frak{a})$ for some
$\frak{a}\in \tilde{{\rm W}}(I,J)$.

Let $S$ be a Serre subcategory of the category of $R$-modules.
Aghapournahr and Melkersson \cite{am}
introduced the notion of $S$-sequences on $M$ as a generalization of regular sequences. Suppose that $S$ is a Melkersson subcategory with respect to
 $I$, $M$ is a $ZD$-module and $M/IM\notin S$.
In \cite[Theorem 2.9]{pl} it is proved
 that  all maximal $S$-sequences on $M$ in $I$, have the same length.
  If this common length is denoted by $S-\depth_I(M)$, then
$S-\depth_I(M)=\inf\{i: H^i_I(M)\notin S\}$; see \cite[Corollary 2.12]{pl}.
In \ref{23}, we generalize this result as follows.
Let $S$ be the class of  Artinian $R$-modules, or  the class of all $R$-modules $N$ with $\dim_RN\leq k$,
where $k$ is an integer.
 Then
$\inf\{S-\depth_\frak{a}(M):
\frak{a}\in \tilde{\rm W}(I,J)\}=\inf\{i: H^{i}_{I,J}(M)\notin S\}$,
where  $M$ is a finitely generated $R$-module, or is a  $ZD$-module
 such that $M/\frak{a}M\notin S$
for all $\frak{a}\in \tilde{\rm W}(I,J)$.

In section 3, we get some identities between local cohomology modules.
Let $\Supp_R H^{i}_{I,J}(M)$ be a finite subset of $\Max(R)$ for
all $i<t$. Then there are maximal ideals $\frak m_1, \frak
m_2,\ldots,\frak m_k$ of $R$ such that
$H^{i}_{I,J}(M)\cong H^{i}_{\frak m_1}(M)\oplus H^{i}_{\frak
m_2}(M)\oplus\cdots\oplus H^{i}_{\frak m_k}(M)$ for all $i<t$; see
\ref{40}. As a consequence  we conclude  that,
if $(R,\frak m)$ is a local ring, then
$\inf\{i: H^{i}_{I,J}(M)\text{is not Artinian}\}
=\inf\{i: H^{i}_{I,J}(M)\not\cong H^{i}_{\frak m }(M)\}$; see \ref{38}.

\section{Extension functors of local cohomology modules}

Recall that $R$ is a Noetherian ring, $I$ and $J$ are ideals of
$R$ and $M$ is an $R$-module.

\begin{definition}
A full subcategory of the category of $R$-modules is said to be
Serre subcategory, if it is closed under taking submodules,
quotients and extensions. A Serre subcategory $S$ is said to be
a Melkersson subcategory with respect to $I$, if for any $I$-torsion $R$-module $M$,
 $0:_MI\in S$ implies that $M\in S$.
A Serre subcategory is called Melkersson subcategory when it is a Melkersson
subcategory with respect to all ideals of $R$.
\end{definition}

The class of finitely generated modules and the class of
 weakly Laskerian modules are Serre subcategories. Aghapournahr and Melkersson \cite[Lemma 2.2]{am}
  proved  that if a Serre subcategory is
closed under taking injective hulls, then it is a Melkersson
subcategory.
The class of zero modules, Artinian $R$-modules,  modules with finite support and the class of $R$-modules
$N$ with $\dim_R N\leq k$, where $k$ is a non-negative integer, are Serre subcategories closed
under taking injective hulls, and hence are Melkersson subcategories; see \cite[Example 2.4]{am}.
The class of $I$-cofinite Artinian
modules  is a Melkersson subcategory with respect to $I$,
but is not closed under taking injective hulls; see \cite[Example 2.5]{am}.

 The following result is a generalization of \cite[Theorem 2.1]{atv}.

\begin{theorem}\label{11}
Let $N$ be an  $(I,J)$-torsion $R$-module. If $\Ext_R^{t-i}(N,
H^{i}_{I,J}(M))\in S$ for all $i\leq t$, then $\Ext_R^t(N, M)\in
S$.
\end{theorem}
\begin{proof}
Let
$F(-)=\Hom_R(N,-)$ and $G(-)=\Gamma_{I,J}(-)$. Then we have
 $FG(M)=
\Hom_R(N, M)$. By \cite[Theorem 11.38]{r}, there is the
Grothendieck  spectral sequence
$$E_2^{p,q}:=\Ext_R^{p}(N, H^{q}_{I,J}(M))\Rightarrow \Ext_R^{p+q}(N,
M).$$  There is a finite filtration
$$0=\phi^{t+1}H^{t}\subseteq\phi^{t}H^{t}\subseteq\cdots
\subseteq\phi^{1}H^{t}\subseteq\phi^{0}H^{t}=\Ext_R^{t}(N,M)$$
such that $E_\infty^{t-i,i}\cong
\phi^{t-i}H^{t}/\phi^{t+1-i}H^{t}$ for all $i\leq t$.
It is enough to show that $\phi^{0}H^t\in S$. By hypothesis,
$E_j^{t-i,i}\in S$ for all $j\geq 2$ and $i\leq t$, and
so $E_{\infty}^{t-i,i}\in S$ for all $i \leq t$. The
 sequence
$$0\longrightarrow\phi^{t+1-i}H^t\longrightarrow\phi^{t-i}H^t\longrightarrow
E_\infty^{t-i,i}\longrightarrow 0$$
is exact for all $i\leq t$.
Therefore it follows that $\phi^{0}H^t\in S$.
 \end{proof}

\begin{corollary}\label{12}
Let $N$ be an $(I,J)$-torsion $R$-module. If $\Ext_R^{j}(N,
H^{i}_{I,J}(M))\in S$ for all $i<t$ and all $j< t-i$, then
$\Ext_R^{i}(N,M)\in S$ for all $i<t$.
\end{corollary}

\begin{corollary}\label{15}
Suppose that $S$ is a Melkersson subcategory with respect to  $I$,
and $N$ is a finitely generated $R$-module with $\Supp_R N=V(\frak{a})$ for some
$\frak{a}\in \tilde{{\rm W}}(I,J)$. If $\Ext_R^{j}(N, H^{i}_{I,J}(M))\in
S$ for all $i<t$ and all $j< t-i$, then $H^{i}_{\frak{a}}(L,M)\in
S$ for all $i<t$ and all finitely generated $R$-modules $L$.
\end{corollary}
\begin{proof}
The result follows by \ref{12} and \cite[Theorem 2.9]{am}.
\end{proof}

The following result improves \cite[Theorem 2.11]{pl2}.

\begin{corollary}\label{16}
Suppose that $S$ is a Melkersson subcategory with respect to $I$, and $N$ is a finitely
generated $R$-module with $\Supp_R N={\rm V}(\frak{a})$ for some
$\frak{a}\in \tilde{{\rm W}}(I,J)$. If $\Ext_R^{j}(N,
H^{i}_{I,J}(M))\in S$ for all $i<t$ and all $j< t-i$, then
$H^{i}_{\frak{a}}(M)\in S$ for all $i<t$.
\end{corollary}
\begin{proof}
In \ref{15}, put $L=R$.
\end{proof}

\begin{corollary}\label{17}
If $S$ is a Melkersson subcategory with respect to $I$, then
$$\inf\{i: H^{i}_{I,J}(M)\notin S\}
\leq\inf\{\inf\{i: H^{i}_{\frak{a}}(M)\notin S\}: \frak{a}\in
\tilde{{\rm W}}(I,J)\}.$$
\end{corollary}

As a generalization of finitely generated modules, Evans \cite{e} introduced $ZD$-modules as follows.

\begin{definition}
An $R$-module $M$ is said to be zero-divisor module ($ZD$-module), if for any submodule
$N$ of $M$, the set $Z_R(M/N)$ is a finite union of prime ideals in $\Ass_R M/N$.
\end{definition}

According to \cite[Example 2.2]{de}, the class  of
$ZD$-modules contains finitely generated, Laskerian,
weakly Laskerian, linearly compact and
Matlis reflexive modules. Also it contains modules whose quotients
have finite Goldie dimension and modules with finite support, in
particular Artinian modules. Therefore the class of $ZD$-modules
 is much larger than that of finitely generated modules.

\begin{definition}
An element $a$ of $R$ is called $S$-regular on $M$, if
$0:_M a\in S$. A sequence $a_1,\ldots, a_t$ is an $S$-sequence on
$M$, if $a_i$ is $S$-regular on $M/(a_1,\ldots, a_{i-1})M$ for $i =
1,\ldots, t$. The $S$-sequence $a_1,\ldots, a_t$  is said to be
maximal $S$-sequence on $M$, if $a_1,\ldots, a_t, y$ is not an
$S$-sequence on $M$ for any $y\in R$.
\end{definition}

When $S$ is the class of zero modules, Artinian $R$-modules, modules with finite support,
and the class of $R$-modules $N$ with $\dim_R N\leq k$, where $k$ is a non-negative integer, then
$S$-sequences on $M$ are, poor M-sequences,
filter-regular sequences,  generalized regular sequences, and $M$-sequences in dimension $>k$, respectively;
see \cite[Example 2.8]{am}.

Let $S$ be a Melkersson subcategory with respect to $I$, and $M$
a $ZD$-module such that $M/IM\notin S$. The current authors, in \cite[Theorem 2.9]{pl}, proved
that all maximal $S$-sequences on $M$ in $I$, have the same length.

\begin{definition}
Let $S$ be a Melkersson subcategory with respect to $I$, and $M$
 a $ZD$-module
such that $M/IM\notin S$. The common length of all maximal
$S$-sequences on $M$ in $I$, is denoted by $S-\depth_{I}(M)$. If
$M/IM\in S$, we set $S-\depth_{I}(M)=\infty$.
\end{definition}

Suppose that $M$ is a $ZD$-module. When $S$ is the class of zero modules, Artinian $R$-modules, and
 modules with finite support, then
$S-\depth_{I}(M)$ is the same as ordinary $\depth_I(M)$, $f-\depth_I(M)$ (filter-depth), and $g-\depth_I(M)$
(generalized depth), respectively.

\begin{corollary}\label{22}
Let $S$ be a Melkersson subcategory with respect to $I$,
 and $M$  a $ZD$-module. Then
$$\inf\{i: H^{i}_{I,J}(M)\notin S\}
\leq\inf\{S-\depth_\frak{a}(M): \frak{a}\in \tilde{\rm
W}(I,J)\}.$$
\end{corollary}
\begin{proof}
The result follows by \ref{17} and \cite[Corollary 2.12]{pl}.
\end{proof}

In the following, we study the relation between
 generalized local cohomology modules and local cohomology modules
 with respect to a pair of ideals.

\begin{corollary}\label{14}
Suppose that  $N$ is a finitely generated $\frak{a}$-torsion $R$-module for
some $\frak{a}\in \tilde{{\rm W}}(I,J)$. If $\Ext_R^{t-i}(N,
H^{i}_{I,J}(M))\in S$ for all $i\leq t$, then
$H^{t}_{\frak{a}}(N,M)\in S$.
\end{corollary}
\begin{proof}
The result follows by \ref{11}.
Note that $\Gamma_\frak{a}(N)\subseteq \Gamma_{I,J}(N)$ and by
\cite[Lemma 2.1]{dst} we have $\Ext_R^{i}(N, M)\cong
H^i_\frak{a}(N, M)$ for any integer $i$.
\end{proof}

\begin{corollary}\label{13}
Suppose that  $N$ is a finitely generated $\frak{a}$-torsion $R$-module for
some $\frak{a}\in \tilde{{\rm W}}(I,J)$. If $\Ext_R^{j}(N,
H^{i}_{I,J}(M))\in S$ for all $i<t$ and all $j< t-i$, then
$H^{i}_{\frak{a}}(N,M)\in S$ for all $i<t$.
\end{corollary}

%\section{Local cohomology and Ext modules}
The following result is a generalization of \cite[Theorem
2.3]{atv}.

\begin{theorem}\label{111}
Let $N$ be an  $(I,J)$-torsion $R$-module. If $\Ext_R^{s+t+1-i}(N,
H^{i}_{I,J}(M))\in S$ for all $i<t$, $\Ext_R^{s+t-1-i}(N,
H^{i}_{I,J}(M))\in S$ for all $t< i< s+t$, and $\Ext_R^{s+t}(N,
M)\in S$, then  $\Ext_R^{s}(N, H^{t}_{I,J}(M))\in S$.
\end{theorem}
\begin{proof}
Let
$F(-)=\Hom_R(N,-)$ and $G(-)=\Gamma_{I,J}(-)$. Then we have $FG(M)=
\Hom_R(N, M)$. By \cite[Theorem 11.38]{r}, there is the
Grothendieck  spectral sequence
$$E_2^{p,q}:=\Ext_R^{p}(N, H^{q}_{I,J}(M))\Rightarrow \Ext_R^{p+q}(N,
M).$$  There is a finite filtration
$$0=\phi^{s+t+1}H^{s+t}\subseteq\phi^{s+t}H^{s+t}\subseteq\cdots
\subseteq\phi^{1}H^{s+t}\subseteq\phi^{0}H^{s+t}=\Ext_R^{s+t}(N,M)$$
such that $E_\infty^{s+t-i,i}\cong
\phi^{s+t-i}H^{s+t}/\phi^{s+t+1-i}H^{s+t}$ for all $i\leq s+t$.
It is enough to show that $E_{2}^{s,t}\in S$.  We have the
following exact sequences
$$0\longrightarrow \Ker d_{t+1-i}^{s,t}\longrightarrow E_{t+1-i}^{s,t}\stackrel{d_{t+1-i}^{s,t}}\longrightarrow
E_{t+1-i}^{s+t+1-i,i}$$ and
$$0\longrightarrow \Im d_{t+1-i}^{s-t-1+i,2t-i}\longrightarrow
\Ker d_{t+1-i}^{s,t}\longrightarrow
E_{t+2-i}^{s,t}\longrightarrow0$$  for all $i$.
By hypothesis,
 $E_{t+1-i}^{s+t+1-i,i}\in S$ for all $i<t$,  and
$E_{1+i-t}^{s+t-1-i,i}\in S$ for all $t<i<s+t$. It follows that
$E_{t+1-i}^{s-t-1+i,2t-i}\in S$ for all $t-s<i<t$. Note that if
$i\leq t-s$, then $E_{t+1-i}^{s-t-1+i,2t-i}=0$. Hence
$E_{t+1-i}^{s-t-1+i,2t-i}\in S$ for all $i<t$, and therefore $\Im
d_{t+1-i}^{s-t-1+i,2t-i}\in S$ for all $i<t$. Also we have
$E_{s+t+2}^{s,t}=E_{\infty}^{s,t}\in S$, because
$E_{\infty}^{s,t}\cong \phi^{s}H^{s+t}/\phi^{s+1}H^{s+t}$ and
$\phi^{s}H^{s+t}\subseteq\phi^{0}H^{s+t}=\Ext_R^{s+t}(N,M)\in S$.
Now the claim follows by the above exact sequences.
 \end{proof}

 \begin{corollary}\label{114}
Let $N$ be an  $(I,J)$-torsion $R$-module. Let $\Ext_R^{j-i}(N,
H^{i}_{I,J}(M))\in S$  for $j=s+t, s+t+1$ and all $i<t$, and
$\Ext_R^{s+t-1-i}(N, H^{i}_{I,J}(M))\in S$ for all $t<i<s+t$. Then
$\Ext_R^{s+t}(N, M)\in S$ if and only if $\Ext_R^{s}(N,
H^{t}_{I,J}(M))\in S$.
\end{corollary}
\begin{proof}
The claim follows by  \ref{11} and  \ref{111}.
\end{proof}

%The following result  is a generalization of \cite[Proposition 3.11]{m}.

The following result is a generalization of \cite[Theorem
2.3]{tp}.

\begin{corollary}\label{116}
Let $N$ be an  $(I,J)$-torsion $R$-module. If $\Ext_R^{t+1-i}(N,
H^{i}_{I,J}(M))\in S$ for all $i<t$, and $\Ext_R^t(N, M)\in S$,
then $\Hom_R(N, H^{t}_{I,J}(M))\in S$.
\end{corollary}
\begin{proof}
In \ref{111}, put $s=0$.
\end{proof}

\begin{proposition}\label{1176}
Let $\Ext_R^{t+1-i}(R/{\frak a},
H^{i}_{I,J}(M))$  be Artinian for all $\frak{a}\in\tilde{{\rm W}}(I,J)$ and all $i<t$,
 and  $\Ext_R^t(R/{\frak a}, M)$ be Artinian
 for all  $\frak{a}\in\tilde{{\rm W}}(I,J)$.
 Then $\Ext_R^{j}(R/\frak{a}, H^{t}_{I,J}(M))$ is Artinian
  for all
   $\frak{a}\in\tilde{{\rm W}}(I,J)$
and all   $j\in \Bbb N_0$.
\end{proposition}
\begin{proof}
By \ref{116}, $\Hom_R(R/\frak a, H^{t}_{I,J}(M))$ is Artinian for all
$\frak{a}\in \tilde{{\rm W}}(I,J)$. Also we know
that
$$H^{t}_{I,J}(M)=\bigcup_{\frak a\in\tilde{{\rm W}}(I,J)}
(0:_{H^{t}_{I,J}(M)} \frak a)= \bigcup_{\frak a\in\tilde{{\rm
W}}(I,J)} \Hom_R(R/\frak a, H^{t}_{I,J}(M)).$$
Now the claim follows by \cite[Theorem 5.1]{m1}
and \cite[Theorem 2.9]{am}.
\end{proof}

\begin{corollary}\label{1180}
 Let  $\Ext_R^i(R/\frak{a}, M)$ be Artinian for
 all $\frak{a}\in \tilde{{\rm W}}(I,J)$ and all $i<t$.
 Then $\Ext_R^{j}(R/\frak{a}, H^{i}_{I,J}(M))$ is Artinian
 for all $\frak{a}\in\tilde{{\rm W}}(I,J)$, all $i<t$ and all $j\in \Bbb N_0$.
\end{corollary}
\begin{proof}
It follows by \ref{1176} that
$\Ext_R^{j}(R/\frak{a}, \Gamma_{I,J}(M))$ is
Artinian for all $\frak{a}\in\tilde{{\rm W}}(I,J)$ and
all $j\in \Bbb N_0$.
Now  again using of \ref{1176}, implies that
$\Ext_R^{j}(R/\frak{a}, H^{1}_{I,J}(M))$ is Artinian for
 all $\frak{a}\in\tilde{{\rm W}}(I,J)$ and
all $j\in \Bbb N_0$.
By continuing this process, the claim follows.
\end{proof}

\begin{lemma}\label{1165}
If $H_{I}^{t-i}(H^{i}_{I,J}(M))\in S$ for all $i\leq t$,
then $H^{t}_{I,J}(M)\in S$.
\end{lemma}
\begin{proof}
Let
$F(-)=\Gamma_{I}(-)$ and $G(-)=\Gamma_{I,J}(-)$. Then $FG(M)=
\Gamma_{I,J}(-)$. The rest of the proof  is similar to that of \ref{11}.
 \end{proof}

\begin{corollary}\label{01}
If $\Ext_R^i(R/\frak{a}, M)$ is Artinian for
 all $\frak{a}\in \tilde{{\rm W}}(I,J)$ and all $i<t$, then
$H^{i}_{I,J}(M)$ is Artinian for all $i<t$.
\end{corollary}
\begin{proof}
It follows by \ref{1180} and \cite[Theorem 2.9]{am} that
$H^{j}_{I}(H^{i}_{I,J}(M))$ is Artinian for all $i<t$ and
all $j\in \Bbb N_0$. Now the claim follows by \ref{1165}.
\end{proof}

%\begin{corollary}\label{02}
 %If  $H^{i}_{\frak{a}}(M)$ is Artinian for
%all $i<t$ and all $\frak{a}\in \tilde{{\rm W}}(I,J)$, then
%$H^{i}_{I,J}(M)$ is Artinian for all $i<t$.
%\end{corollary}
%\begin{proof}
%The claim follows by \ref{01} and \cite[Theorem 2.9]{am}.
%\end{proof}

\begin{proposition}\label{117}
Let $S$ be the class of all $R$-modules $N$ with $\dim_RN\leq k$,
where $k$ is an integer. Let $\Ext_R^{t+1-i}(R/{\frak a},
H^{i}_{I,J}(M))\in S$ for  all $\frak{a}\in \tilde{{\rm W}}(I,J)$ and all $i<t$,
and $\Ext_R^t(R/{\frak a}, M)\in S$ for all
$\frak{a}\in \tilde{{\rm W}}(I,J)$. Then $H^{t}_{I,J}(M)\in S$.
\end{proposition}
\begin{proof}
By \ref{116}, $\Hom_R(R/\frak a, H^{t}_{I,J}(M))\in S$ for all
$\frak{a}\in \tilde{{\rm W}}(I,J)$. Also we know
that
$$H^{t}_{I,J}(M)=\bigcup_{\frak a\in\tilde{{\rm W}}(I,J)}
(0:_{H^{t}_{I,J}(M)} \frak a)= \bigcup_{\frak a\in\tilde{{\rm
W}}(I,J)} \Hom_R(R/\frak a, H^{t}_{I,J}(M)).$$  It follows that
$H^{t}_{I,J}(M)\in S$.
\end{proof}

\begin{corollary}\label{118}
Let $S$ be the class of all $R$-modules $N$ with $\dim_RN\leq k$,
where $k$ is an integer. If  $\Ext_R^i(R/\frak{a}, M)\in S$ for
 all $\frak{a}\in \tilde{{\rm W}}(I,J)$ and all $i<t$, then
$H^{i}_{I,J}(M)\in S$ for all $i<t$.
\end{corollary}
\begin{proof}
We know that
$$\Gamma_{I,J}(M)=\bigcup_{\frak a\in\tilde{{\rm
W}}(I,J)} \Hom_R(R/\frak a, \Gamma_{I,J}(M))= \bigcup_{\frak a\in\tilde{{\rm
W}}(I,J)} \Hom_R(R/\frak a, M).$$
Therefore $\Gamma_{I,J}(M)\in S$.
It follows by \ref{117} that
$H^{1}_{I,J}(M)\in S$.
%Now by again using of  \ref{117}, we have $H^{2}_{I,J}(M)\in S$.
By
keeping this process, the claim follows.
\end{proof}

\begin{corollary}\label{119}
Let $S$ be the class of  Artinian $R$-modules, or  the class of all $R$-modules $N$ with $\dim_RN\leq k$,
where $k$ is an integer. If $H^{i}_{\frak{a}}(M)\in S$
for  all $\frak{a}\in \tilde{{\rm W}}(I,J)$ and all $i<t$, then
$H^{i}_{I,J}(M)\in S$ for  all $i<t$.
\end{corollary}
\begin{proof}
The claim follows by \ref{01}, \ref{118} and \cite[Theorem 2.9]{am}.
\end{proof}

%The following result is one of the main results of this paper.

\begin{corollary}\label{120}
Let $S$ be the class of  Artinian $R$-modules, or  the class of all $R$-modules $N$ with $\dim_RN\leq k$,
where $k$ is an integer. Then the following statements are
equivalent:
\begin{itemize}
\item[(i)] $H^{i}_{I,J}(M)\in S$ for all $i<t$;
\item[(ii)] $H^{i}_{\frak{a}}(M)\in S$ for
all $\frak{a}\in \tilde{{\rm W}}(I,J)$ and all $i<t$.
\end{itemize}
\end{corollary}
\begin{proof}
The claim follows by \ref{16} and \ref{119}.
\end{proof}

\begin{corollary}\label{21}
Let $S$ be the class of  Artinian $R$-modules, or  the class of all $R$-modules $N$ with $\dim_RN\leq k$,
where $k$ is an integer. Then
$$\inf\{i: H^{i}_{I,J}(M)\notin
S\}=\inf\{\inf\{i: H^{i}_{\frak a}(M)\notin S\}: \frak{a}\in
\tilde{{\rm W}}(I,J)\}.$$
\end{corollary}

The following result is a generalization of \cite[Theorem 2.13]{pl}.

\begin{theorem}\label{23}
Let $S$ be the class of  Artinian $R$-modules, or  the class of all $R$-modules $N$ with $\dim_RN\leq k$,
where $k$ is an integer.
Let $M$ be a finitely generated $R$-module, or be a  $ZD$-module
 such that $M/\frak{a}M\notin S$
for all $\frak{a}\in \tilde{\rm W}(I,J)$.
 Then
$$\inf\{i: H^{i}_{I,J}(M)\notin S\}=\inf\{S-\depth_\frak{a}(M):
\frak{a}\in \tilde{\rm W}(I,J)\}.$$
\end{theorem}
\begin{proof}
The claim follows by \ref{21} and \cite[Theorem 2.13]{pl}.
\end{proof}

\section{Some identities between local cohomology modules }

Suppose that
$$E^\bullet(M): 0\longrightarrow
E^0_R(M)\stackrel{d^0}\longrightarrow E^1_R(M)\longrightarrow
\cdots \longrightarrow E^i_R(M)\stackrel{d^i}\longrightarrow
E^{i+1}_R(M)\longrightarrow\cdots~(*)$$
 is a minimal injective resolution
of $M$, where
$E^i_R(M)\cong\bigoplus_{\frak p\in {\rm Spec(R)}}\mu^i({\frak
p},M)E_R(R/\frak p)$ is a decomposition of $E^i_R(M)$ as the
direct sum of indecomposable injective $R$-modules,
$E_R(R/\frak p)$ denotes the injective hull of $R/\frak p$ and
$\mu^i({\frak p},M)$ denotes the $i$-th Bass number of $M$ with
respect to $\frak p$. It follows by \cite[Proposition 1.11]{ty} that
$$\Gamma_{I,J}(E^i_R(M))\cong\bigoplus_{\frak p\in {\rm
W}(I,J)}\mu^i({\frak p},M)E_R(R/\frak p).$$
Hence
$\Supp_R \Gamma_{I,J}(E^i_R(M))=\{\frak p\in  {\rm W}(I,J):
\mu^i({\frak p},M)\not=0\}.$

The above mentioned results are assumed known through this section.
%The following result is a generalization of \cite[Theorem 5.5]{m}.

\begin{theorem}\label{29}
Let $S$ be a Serre subcategory  closed under taking injective
hulls. The following conditions are equivalent:
\begin{itemize} \item[(i)] $ H^{i}_{I,J}(M)\in
S$ for all $i<t$. \item[(ii)]$ \Gamma_{I,J}(E^i_R(M))\in S$ for
all $i<t$.
\end{itemize}
\end{theorem}
\begin{proof}
Since
$\Gamma_{I,J}(E^i_R(M))$ is injective and $\Ker
\Gamma_{I,J}(d^i)=\Ker d^i \cap \Gamma_{I,J}(E^i_R(M))$, thus
$\Gamma_{I,J}(E^i_R(M))$ is injective hull of $\Ker
\Gamma_{I,J}(d^i)$. Now the claim follows by \cite[Lemma 5.4]{m}.
We note that the proof of \cite[Lemma 5.4]{m} is still valid if the class of Artinian $R$-modules
is replaced by a Serre subcategory that is closed under taking injective
hulls.
\end{proof}

\begin{corollary}\label{31}
The following statements are equivalent:
\begin{itemize}
\item[(i)] $\Supp_R H^{i}_{I,J}(M)$ is a finite subset of $\Max(R)$ for all
$i<t$;
\item[(ii)] $H^{i}_{I,J}(M)$ is Artinian for all $i<t$.
\end{itemize}
\end{corollary}
%\begin{proof}
%We note that $\Supp_R \Gamma_{I,J}(E^i_R(M))=\{\frak p\in  {\rm W}(I,J):
%\mu^i({\frak p},M)\not=0\}.$
%\end{proof}

%The following result is a generalization of \cite[Proposition 2.8]{pl}.

\begin{corollary}\label{319}
If $(R, \frak m)$ is a local ring, then
$$\inf\{i:
\Supp_R H^{i}_{I,J}(M)\not\subseteq\{\frak m\}\}=
\inf\{i: H^{i}_{I,J}(M) \text{is not Artinian} \}.$$
\end{corollary}

\begin{proposition}\label{32}
Let $M$ be a finitely generated $R$-module, or be a $ZD$-module such that
$M_{\frak{p}}/{\frak{p}M_{\frak{p}}}\neq0$ for all $\frak{p}\in
{\rm W}(I,J)$. Then
$$\inf\{i: H^{i}_{I,J}(M)\neq0\}=\inf\{\depth M_{\frak p}: {\frak p}\in {\rm
W}(I,J)\}.$$
\end{proposition}
\begin{proof}
Let $t=\inf\{\depth M_{\frak p}: {\frak p}\in {\rm W}(I,J)\}$.
It follows by \cite[Corollary 2.14]{pl} that
$\mu^{i}(\frak p,M)=0$ for all $\frak p\in {\rm W}(I,J)$ and
all $i<t$. So $\Gamma_{I,J}(E^i_R(M))=0$ for all $i<t$,
and hence $H^{i}_{I,J}(M)=0$ for all $i<t$.
Therefore $t\leq\inf\{i: H^{i}_{I,J}(M)\neq0\}$. Now it is enough to show that
$H^{t}_{I,J}(M)\neq0$. By assumption, there is  ${\frak q}\in {\rm
W}(I,J)$ such that $t=\depth M_{\frak q}$. It follows by \cite[Corollary 2.14]{pl} that
 $\mu^{t}(\frak q,M)\not=0$. Therefore $\Gamma_{I,J}(E^t_R(M))\not=0$, and hence
$H^t_{I,J}(M)\not=0$ by \ref{29}.
\end{proof}

We can get a generalization of   \cite[Theorem 2.4]{cw}.

\begin{proposition}\label{33}
Let $M$ be a finitely generated $R$-module, or be a $ZD$-module such that
$M_{\frak{p}}/{\frak{p}M_{\frak{p}}}\neq0$ for all $\frak{p}\in
{\rm W}(I,J)-\Max(R)$. Then
$$\inf\{i: \Supp_R H^{i}_{I,J}(M)\not\subseteq \Max(R) \}=
\inf\{\depth M_{\frak p}: {\frak p}\in {\rm W}(I,J)-\Max(R)\}.$$
\end{proposition}
\begin{proof}
Let $t=\inf\{i: \Supp_R H^{i}_{I,J}(M)\not\subseteq \Max(R) \}$.
It follows by \ref{29} that $\mu^{i}(\frak p,M)=0$ for all $\frak
p\in {\rm W}(I,J)-\Max(R)$ and  all $i<t$, and there is  $\frak
q\in {\rm W}(I,J)-\Max(R)$ such that $\mu^{t}(\frak q,M)\neq 0$.
Now it follows by \cite[Corollary 2.14]{pl} that $\depth M_{\frak p}\geq t$ for all $\frak p\in {\rm
W}(I,J)-\Max(R)$, and $\depth M_{\frak q}=t$.
Therefore
$\inf\{\depth M_{\frak p}: \frak p\in {\rm W}(I,J)-\Max(R)\}=t.$
\end{proof}

\begin{theorem}\label{35}
Let $\Supp_R H^{i}_{I,J}(M)$ be a finite subset of $\Max(R)$ for
all $i<t$. Then there are maximal ideals $\frak m_1, \frak
m_2,\ldots,\frak m_k\in {\rm W}(I,J)$ such that
$H^{i}_{I,J}(M)\cong H^{i}_{\frak m_1\frak m_2\cdots\frak m_k}(M)$
for all $i<t$.
\end{theorem}
\begin{proof}
It follows by \ref{29} that $\Supp_R \Gamma_{I,J}(E^i(M))$ is a
finite subset of $\Max(R)$ for all $i<t$. Let $\Supp_R
\Gamma_{I,J}(E^i(M))=\{\frak m_{i1},\frak m_{i2},...,\frak
m_{ik_i}\}$, where $k_i$ is an integer. Then
$$\Gamma_{I,J}(E^i(M))\cong \bigoplus_{j=1}^{k_i} \mu^{i}(\frak
m_{ij},M) E_R(R/\frak m_{ij})$$ for all $i<t$. Put $\frak
a=\prod_{i,j} \frak m_{ij}$. Then ${\rm V}(\frak a)=\{\frak m_{ij}:
0\leq i<t, 1\leq j\leq k_i\}$, and $\frak a\in \tilde{\rm W}(I,J)$.
Therefore
\begin{eqnarray*}
\Gamma_{\frak a}(E^i(M))&\cong& \bigoplus_{\frak p\in {\rm
V}(\frak a)} \mu^{i}(\frak p,M) E_R(R/\frak p)\\
&=&\bigoplus_{j=1}^{k_i} \mu^{i}(\frak m_{ij},M) E_R(R/\frak
m_{ij})\cong\Gamma_{I,J}(E^i(M))
\end{eqnarray*}
 for all $i<t$, and hence
$H^{i}_{I,J}(M)\cong H^{i}_{\frak a}(M)$ for all $i<t$.
\end{proof}

\begin{corollary}\label{40}
Let $\Supp_R H^{i}_{I,J}(M)$ be a finite subset of $\Max(R)$ for
all $i<t$. Then there are maximal ideals $\frak m_1, \frak
m_2,\ldots,\frak m_k\in {\rm W}(I,J)$ such that
$$H^{i}_{I,J}(M)\cong H^{i}_{\frak m_1}(M)\oplus H^{i}_{\frak
m_2}(M)\oplus\cdots\oplus H^{i}_{\frak m_k}(M)$$ for all $i<t$.
\end{corollary}
\begin{proof}
The claim follows by \ref{35} and the Mayer-Vietoris sequence \cite[3.2.3]{bs}.
\end{proof}

\begin{corollary}\label{36}
Let $(R,\frak m)$ be  a local ring. If
$\Supp_RH^{i}_{I,J}(M)\subseteq\{\frak m\}$ for all $i<t$, then
$H^{i}_{I,J}(M)\cong H^{i}_{\frak m }(M)$ for all $i<t$.
\end{corollary}

\begin{corollary}\label{37}
If $(R,\frak m)$ is a local ring, then
$$\inf\{i: \Supp_RH^{i}_{I,J}(M)\not\subseteq\{\frak m\}\}
=\inf\{i: H^{i}_{I,J}(M)\not\cong H^{i}_{\frak m }(M)\}.$$
\end{corollary}

The following result is a generalization of \cite[Proposition
2.5]{cw}.

\begin{corollary}\label{38}
If $(R,\frak m)$ is a local ring, then
$$\inf\{i: H^{i}_{I,J}(M)\text{is not Artinian}\}
=\inf\{i: H^{i}_{I,J}(M)\not\cong H^{i}_{\frak m }(M)\}.$$
\end{corollary}
\begin{proof}
The claim follows by
\ref{319} and \ref{37}.
\end{proof}

%The following result is a generalization of \cite[Corollary
%2.13]{pl2}.

\begin{corollary}\label{39}
Let $(R,\frak m)$ be a local ring. If
$\Supp_RH^{i}_{I,J}(M)\subseteq\{\frak m\}$ for all $i<t$, then
$H^{i}_{I,J}(M)\cong H^{i}_{\frak a}(M)$ for all $\frak a\in
\tilde{\rm W}(I,J)$ and all $i<t$.
\end{corollary}
\begin{proof}
It follows by  \ref{16} that  $\Supp_RH^{i}_{\frak
a}(M)\subseteq\{\frak m\}$ for all $\frak a\in \tilde{\rm W}(I,J)$
and all $i<t$.  Therefore
$H^{i}_{\frak a}(M)\cong H^{i}_{\frak m}(M)$ for all $\frak a\in \tilde{\rm W}(I,J)$ and
all $i<t$, by  \ref{36}.  Now again using of  \ref{36} implies that
$H^{i}_{I,J}(M)\cong H^{i}_{\frak m }(M)$ for all $i<t$, and the claim follows.
\end{proof}

%\section{Acknowledgment}
%The authors are deeply grateful to the referee  for his/her
%careful reading of the manuscript and very helpful suggestions.

\end{document}